\documentclass[10pt,reqno]{amsart}
\usepackage{amsmath,amsthm,amssymb,latexsym,esint,cite,mathrsfs}
\usepackage{verbatim,wasysym,cite,relsize,color}
\usepackage[latin1]{inputenc}
\usepackage{microtype}
\usepackage{color,enumitem,graphicx}
\usepackage[colorlinks=true,urlcolor=blue, citecolor=red,linkcolor=blue,
linktocpage,pdfpagelabels, bookmarksnumbered,bookmarksopen]{hyperref}
\usepackage[hyperpageref]{backref}
\usepackage[english]{babel}
\usepackage{tikz}
\usepackage[font=small,skip=5pt]{caption}
\usepackage{geometry}
\numberwithin{equation}{section}
\theoremstyle{plain}
\newtheorem{theorem}{Theorem}[section]
\newtheorem{lemma}[theorem]{Lemma}

\newtheorem{corollary}[theorem]{Corollary}
\theoremstyle{Remark} 
\newtheorem{remark}{Remark}[section]

\newcommand{\R}{\mathbb R}
\newcommand{\C}{\mathbb C}

\newcommand{\vareps}{\varepsilon}

\DeclareMathOperator*{\loc}{loc}

\DeclareMathOperator*{\thes}{th}

\DeclareMathOperator*{\ima}{Im}
\DeclareMathOperator*{\rea}{Re}

\title[Blow-up for linearly damped NLS]
{Blow-up criteria for linearly damped nonlinear Schr\"odinger equations}
\author[V. D. Dinh]{Van Duong Dinh}
\address[V. D. Dinh]{Laboratoire Paul Painlev\'e UMR 8524, Universit\'e de Lille CNRS, 59655 Villeneuve d'Ascq Cedex, France
and 
Department of Mathematics, HCMC University of Pedagogy, 280 An Duong Vuong, Ho Chi Minh, Vietnam}
\email{contact@duongdinh.com}


\begin{document}
	
	\begin{abstract}
	We consider the Cauchy problem for linearly damped nonlinear Schr\"odinger equations
	\[
	i\partial_t u + \Delta u + i a u= \pm |u|^\alpha u, \quad (t,x) \in [0,\infty) \times \R^N,
	\]
	where $a>0$ and $\alpha>0$. We prove the global existence and scattering for a sufficiently large damping parameter in the energy-critical case. We also prove the existence of finite time blow-up $H^1$ solutions to the focusing problem in the mass-critical and mass-supercritical cases. 
	\end{abstract}

	\maketitle

	\section{Introduction}
	\label{S1}
	\setcounter{equation}{0}
	We consider the Cauchy problem for linearly damped nonlinear Schr\"odinger equations
	\begin{equation} \label{DNLS}
		\left\{ 
		\begin{array}{rcl}
			i\partial_t u + \Delta u + i a u &=& \mu |u|^\alpha u, \quad (t,x) \in [0,\infty) \times \R^N, \\
			u(0,x)&=& u_0(x),
		\end{array}
		\right.
	\end{equation}
	where $u: [0,\infty) \times \mathbb{R}^N \rightarrow \mathbb{C}$, $u_0: \mathbb{R}^N \rightarrow \mathbb{C}$, $N\geq 1, a>0$, $\mu \in \{\pm 1\}$ and $\alpha>0$. The case $\mu=1$ (resp. $\mu=-1$) corresponds to the defocusing (resp. focusing) case. The linearly damped nonlinear Schr\"odinger equation appears in various areas of nonlinear optics, plasma physics and fluid mechanics. It has been studied by many mathematicians and physicists (see e.g. \cite{ADKM, CZW, Darwich, GRH, Fibich, Fibich-book, Inui, OT,  PPV, RBC, Tsutsumi-M, Tsutsumi-M-90}). 
	
	The equation \eqref{DNLS} is locally well-posed in $H^1$ (see e.g \cite{Cazenave}). More precisely, for any $u_0 \in H^1$, there exist $T^* \in (0,\infty]$ and a unique solution $u$ to \eqref{DNLS} such that $u\in C([0,T^*), H^1)$. Moreover, if $T^*<\infty$, then $\lim_{t \rightarrow T^*} \|u(t)\|_{H^1}=\infty$. Moreover, local solutions satisfy
	\begin{align} \label{mass-u}
	\|u(t)\|_{L^2} = e^{-at} \|u_0\|_{L^2}
	\end{align} 
	and
	\begin{align} \label{energy-u}
	\frac{d}{dt} E(u(t)) = - a K(u(t))
	\end{align}
	for any $t\in [0,T^*)$, where
	\begin{align} \label{defi-K}
	\begin{aligned}
	E(u(t)) &:=\frac{1}{2}\|\nabla u(t)\|_{L^2} + \frac{\mu}{\alpha+2} \|u(t)\|^{\alpha+2}_{L^{\alpha+2}}, \\
	K(u(t)) &:= \|\nabla u(t)\|^2_{L^2} + \mu \|u(t)\|^{\alpha+2}_{L^{\alpha+2}}.
	\end{aligned}
	\end{align}
	
	\subsection{Known results}
	Let us recall some known results related to \eqref{DNLS}. 
	
	{\bf $\blacktriangleright$ The defocusing case $\mu=1$.} Using \eqref{mass-u}, \eqref{energy-u} and the blow-up alternative, it is easy to see that local solutions can be extended globally in time, i.e. $T^*=\infty$ if $0<\alpha<\alpha^*$, where
	\begin{align} \label{defi-alph-star}
	\alpha^*:=\left\{
	\begin{array}{cl}
	\frac{4}{N-2} &\text{if } N\geq 3, \\
	\infty &\text{if } N=1,2.
	\end{array}
	\right.
	\end{align}
	
	Recently,  Inui \cite{Inui} proved that all global solutions to the defocusing \eqref{DNLS} with $0<\alpha<\alpha^*$ scatter exponentially in the sense that there exists $u_+\in H^1$ such that 
	\begin{align} \label{defi-expo-scat}
	\lim_{t\rightarrow \infty} e^{at} \|u(t) - e^{-at} e^{it\Delta} u_+\|_{H^1}=0.
	\end{align}
	
	{\bf $\blacktriangleright$ The focusing case $\mu=-1$,}	$\bullet$ In the case $0<\alpha<\frac{4}{N}$, using \eqref{mass-u} and the following blow-up alternative (see \cite{Tsutsumi-Y} or \cite{Cazenave}): if $T^*<\infty$, then $\lim_{t\rightarrow T^*} \|u(t)\|_{L^2}=\infty$, we see that local solutions can be extended globally in time. Moreover, Inui \cite{Inui} proved that these global solutions scatter exponentially in the sense of \eqref{defi-expo-scat}. This fact shows an interesting effect of the linear damping to the nonlinear Schr\"odinger equation. Note that global solutions to the undamped NLS (i.e. $a=0$) does not scatter to the free solution when $\alpha \leq \frac{2}{N}$. 
	
	$\bullet$ In the case $\alpha=\frac{4}{N}$, Ohta-Todorova \cite{OT} proved that for any $u_0 \in H^1$, there exists $a^*=a^*(\|u_0\|_{H^1})>0$ such that for all $a>a^*$, the corresponding solution to \eqref{DNLS} exists globally in time.
		
	Darwich \cite{Darwich} proved that if $u_0 \in H^1$ satisfies $\|u_0\|_{L^2}< \|Q\|_{L^2}$, where $Q$ is the unique positive radial solution to 
	\begin{align} \label{ell-equ-mass}
	-\Delta Q + Q - |Q|^{\frac{4}{N}} Q=0,
	\end{align}
	then for any $a>0$, the corresponding solution to \eqref{DNLS} exists globally in time. 
	
	Inui \cite{Inui} proved that the global solution obtained by Darwich \cite{Darwich} actually scatters exponentially in the sense of \eqref{defi-expo-scat}. Moreover, he also proved that all global solutions to the focusing problem \eqref{DNLS} with $\alpha=\frac{4}{N}$ scatter exponentially in the sense of \eqref{defi-expo-scat}. 
	
	Fibich \cite{Fibich} provided some numerical simulations which suggest the existence of finite time blow-up solutions to \eqref{DNLS}. Recently, Darwich \cite{Darwich} proved in dimensions $N\leq 4$ the existence of log-log speed blow-up solutions to \eqref{DNLS} with $\|u_0\|_{L^2} = \|Q\|_{L^2} + \delta$ for some $\delta>0$ sufficiently small. The proof of this result is based on the geometric decomposition technique using a blow-up result of Merle-Raphael \cite{MR}. Note that a general criterion for the existence of finite time blow-up solutions to the focusing problem \eqref{DNLS} in the mass-critical case remains an open problem.  
	
	One can infer from the results of Ohta-Todorova \cite{OT}, Inui \cite{Inui} and Darwich \cite{Darwich} that there are only two type of solutions to the focusing problem \eqref{DNLS} in the mass-critical case: finite time blow-up solutions and global scattering solutions. There is no global solution which does not scatter. This is another difference from the undamped NLS where there exist global non scattering solutions called standing waves solutions. Note that the non existence of standing waves solutions of the form $u(t,x) = e^{i\omega t} \phi(x)$ with $\omega \in \R$ can be easily seen from \eqref{mass-u} and the fact $\|u(t)\|_{L^2}=\|\phi\|_{L^2}$. 
	
	This is observed through numerical simulations (see \cite{Fibich}) that for a given initial data condition that leads to blow-up in the undamped NLS, there is a threshold value $a_{\thes}$, which depends on the initial data condition, such that collapse is arrested when $a>a_{\thes}$ and a singularity forms when $a<a_{\thes}$. However, a rigorous proof for this result still remains open. 
	
	$\bullet$ In the case $\frac{4}{N}<\alpha<\alpha^*$, Tsutsumi \cite{Tsutsumi-M} proved that if $u_0 \in \Sigma:= H^1 \cap L^2(|x|^2 dx)$ satisfies
	\[
	E(u_0) \leq 0, \quad \frac{a\alpha}{N\alpha-4} I(u_0) + V(u_0)<0,
	\]
	where
	\begin{align} \label{defi-I-V}
	I(u_0):= \|xu_0\|^2_{L^2}, \quad V(u_0):= \int x \cdot \ima \left(\nabla u_0 \overline{u}_0\right) dx,
	\end{align}
	then there exists $a_* = a_*(\|u_0\|_{H^1})>0$ such that for all $0<a<a_*$ the corresponding solution to \eqref{DNLS} blows up in finite time. 
	
	Ohta-Todorova \cite{OT} improved Tsutsumi's results and showed that if $u_0 \in \Sigma$ satisfies one of the following conditions:
	\begin{itemize}
		\item $E(u_0)<0$,
		\item $E(u_0)=0$ and $V(u_0)<0$,
		\item $E(u_0)>0$ and $V(u_0) + \sqrt{2E(u_0) I(u_0)} <0$,
	\end{itemize}
	then there exists $a_*=a_*(\|u_0\|_{H^1})>0$ such that for all $0<a<a_*$, the corresponding solution blows up in finite time. They also proved that for any $u_0 \in H^1$, there exists $a^*=a^*(\|u_0\|_{H^1})>0$ such that for all $a>a^*$, the corresponding solution to \eqref{DNLS} exists globally in time. Moreover, they constructed invariant sets under the flow of \eqref{DNLS} which does not depend on the damping parameter $a$ and showed the global existence for initial data in these sets. More precisely, they defined
	\[
	\mathcal{A}_\omega:= \left\{ \phi \in H^1 \backslash \{0\} \ : \ S_\omega(\phi) <d(\omega), K_\omega(\phi)>0\right\},
	\]
	where
	\begin{align*}
	S_\omega(\phi) &:=\frac{1}{2}\|\nabla \phi\|^2_{L^2} +\frac{\omega}{2} \|\phi\|^2_{L^2} -\frac{1}{\alpha+2}\|\phi\|^{\alpha+2}_{L^{\alpha+2}}, \\
	K_\omega(\phi) &:= \|\nabla \phi\|^2_{L^2} +\omega \|\phi\|^2_{L^2} - \|\phi\|^{\alpha+2}_{L^{\alpha+2}},
	\end{align*}
	and
	\begin{align*}
	d(\omega) := \inf \left\{ S_\omega(\phi) \ : \ \phi \in H^1 \backslash \{0\}, K_\omega(\phi)=0 \right\}.
	\end{align*}
	It was shown in \cite{OT} that if $u_0 \in \bigcup_{\omega>0} \mathcal{A}_\omega$, then for all $a>0$, the corresponding solution to \eqref{DNLS} exists globally in time. A similar result with $\omega=1$ was shown by Chen-Zhang-Wei in \cite{CZW}. 
	
	Inui \cite{Inui} also proved that if $u$ is a global $H^1$ solution to \eqref{DNLS} satisfying 
	\begin{align*}
	\left\{
	\renewcommand*{\arraystretch}{1.3}
	\begin{array}{ll}
	\lim_{t\rightarrow \infty} e^{-\frac{4-(N-2)\alpha}{N\alpha-4} a t} \|\nabla u(t)\|_{L^2} =0 &\text{if } \frac{4}{N-1}\leq \alpha <\alpha^*, \\
	\lim_{t\rightarrow \infty} e^{-at} \|\nabla u(t)\|_{L^2} =0 &\text{if } \frac{4}{N}<\alpha <\frac{4}{N-1},
	\end{array}
	\right.
	\end{align*}
	then $u$ scatters exponentially in the sense of \eqref{defi-expo-scat}. This implies in particular that if $u$ is a global $H^1$ solution to \eqref{DNLS} satisfying $\sup_{t\in [0,\infty)} \|u(t)\|_{H^1} <\infty$, then $u$ scatters exponentially. Using this fact and a result of \cite{OT}, he inferred that for $\frac{4}{N}<\alpha<\alpha^*$, if $u_0 \in \bigcup_{\omega>0} \mathcal{A}_\omega$, then the corresponding solution to \eqref{DNLS} scatters exponentially.
	
	\subsection{Main results}
	Complementing to aforementioned results, the main purpose of this paper is twofold. 
	\begin{itemize}
		\item We show the global existence and scattering for a sufficiently large damping parameter in the energy-critical case. 
		\item We prove the existence of finite time blow-up solutions for the focusing problem \eqref{DNLS} in the mass-critical and mass-supercritical cases.
	\end{itemize}

	Our first result concerns the global existence and scattering for \eqref{NLS} in the energy-critical case $\alpha=\frac{4}{N-2}$. 

	\begin{theorem} \label{theo-scat-ener}
		Let $N\geq 3$, $\alpha=\frac{4}{N-2}$ and $\mu\in \{\pm 1\}$. Let $u_0 \in H^1$. Then there exists $a^*=a^*(u_0)>0$ such that for all $a>a^*$, the corresponding solution to \eqref{DNLS} exists globally in time and scatters exponentially in the sense of \eqref{defi-expo-scat}.
	\end{theorem}
	
	\begin{remark}
	This result shows an interesting effect of the linear damping term since the global well-posedness and scattering for the energy-critical NLS without damping are much more difficult to prove (see e.g. \cite{KM, CKSTT}). 
	\end{remark}
	
	Our next result is the following blow-up criteria for the focusing problem \eqref{NLS} in the mass-critical case.
	\begin{theorem} \label{theo-blow-mass-sigma}
		Let $N\geq 1$, $\alpha=\frac{4}{N}$ and $\mu=-1$. If $u_0 \in \Sigma$ satisfies one of the following conditions:
		\begin{itemize}
			\item $E(u_0)<0$, 
			\item $E(u_0) =0$ and $V(u_0)<0$, 
			\item $E(u_0)>0$ and $V(u_0) + \sqrt{2E(u_0) I(u_0)} <0$,
		\end{itemize}
		where $I$ and $V$ are as in \eqref{defi-I-V}, then there exists $a_*=a_*(\|u_0\|_{H^1})>0$ such that for all $0<a<a_*$, the corresponding solution to \eqref{DNLS} blows up in finite time.  
	\end{theorem}
	
	\begin{remark}
		This result gives a rigorous proof for the existence of a threshold value for the mass-critical focusing problem \eqref{DNLS} observed by Fibich \cite{Fibich}.
	\end{remark}

	We also have the existence of finite time blow-up solutions with radially symmetric initial data. To state our result, we introduce the following function
	\begin{align} \label{defi-vartheta}
	\vartheta(r):= \left\{
	\begin{array}{ccc}
	2r &\text{if}& 0\leq r \leq 1, \\
	2[r-(r-1)^3] &\text{if}& 1<r\leq 1+1/\sqrt{3}, \\
	\vartheta' <0 &\text{if}& 1+1/\sqrt{3}<r<2, \\
	0 &\text{if}& r\geq 2,
	\end{array}
	\right.
	\end{align} 
	and
	\begin{align} \label{defi-theta}
	\theta(r):=\int_0^r \vartheta(s)ds.
	\end{align}
	For $R>0$ sufficiently large, we define the radial function
	\begin{align} \label{defi-chi-R}
	\chi_R(x)=\chi_R(r):= R^2 \theta(r/R), \quad r=|x|.
	\end{align}
	Denote
	\begin{align} \label{defi-J-W}
	J(u_0):= \int \chi_R|u_0|^2 dx, \quad W(u_0):= \int \nabla \chi_R \cdot \ima \left( \nabla u_0 \overline{u}_0\right) dx.
	\end{align}
	
	\begin{theorem} \label{theo-blow-mass-rad}
		Let $N\geq 2$, $\alpha=\frac{4}{N}$ and $\mu=-1$. If $u_0 \in H^1$ is radially symmetric and satisfies one of the following conditions:
		\begin{itemize}
			\item $E(u_0)<0$, 
			\item $E(u_0) =0$ and $W(u_0)<0$, 
			\item $E(u_0)>0$ and $W(u_0) + \sqrt{8E(u_0) J(u_0)} <0$,
		\end{itemize}
		then there exists $a_*=a_*(\|u_0\|_{H^1})>0$ such that for all $0<a<a_*$, the corresponding solution to \eqref{DNLS} blows up in finite time. 
	\end{theorem}
	
	Our next results concern the existence of finite time blow-up solutions for the focusing problem \eqref{DNLS} in the mass-supercritical and energy-subcritical case.
	\begin{theorem} \label{theo-blow-super-sigma}
		Let $N\geq 1$, $\frac{4}{N}<\alpha<\alpha^*$ and $\mu=-1$. If $u_0 \in \Sigma$ satisfies one of the following conditions:
		\begin{itemize}
			\item $E(u_0)<0$,
			\item $E(u_0) =0$ and $V(u_0)<0$,
			\item $E(u_0)>0$ and $V(u_0) + \sqrt{2 E(u_0) I(u_0)}<0$,
		\end{itemize}
		where $I$ and $V$ are as in \eqref{defi-I-V}, then there exists $a_*=a_*(\|u_0\|_{H^1})>0$ such that for all $0<a<a_*$, the corresponding solution to \eqref{DNLS} blows up in finite time.
	\end{theorem}
	
	\begin{remark}
		Corollary $\ref{theo-blow-super-sigma}$ was proved in \cite{OT} by using virial identities related to \eqref{DNLS}. Here we give a simple proof. 
	\end{remark}
	
	We next introduce $\theta: [0,\infty) \rightarrow [0,\infty)$ satisfying
	\begin{align} \label{defi-theta-super}
	\theta(r):= \left\{
	\begin{array}{ccc}
	r^2 &\text{if} & 0\leq r \leq 1, \\
	2 &\text{if} & r\geq 2,  
	\end{array}
	\right. \quad \text{and}\quad \theta''(r) \leq 2 \text{ for } r\geq 0.
	\end{align}
	Note that the function $\theta$ defined in \eqref{defi-theta} satisfies the above conditions. For $R>0$ sufficiently large, we define the radial function
	\begin{align} \label{defi-chi-R-super}
	\chi_R(x)=\chi_R(r):=R^2 \theta(r/R), \quad r = |x|.
	\end{align}
	
	\begin{theorem} \label{theo-blow-super-rad}
		Let $\mu=-1$ and
		\[
		\left\{
		\begin{array}{ccc}
		\frac{4}{N}<\alpha<\frac{4}{N-2} &\text{if}& N\geq 3, \\
		2<\alpha \leq 4 &\text{if} & N=2.
		\end{array}
		\right.
		\] 
		If $u_0 \in H^1$ is radially symmetric and satisfies one of the following conditions:
		\begin{itemize}
			\item $E(u_0)<0$,
			\item $E(u_0) =0$ and $W(u_0)<0$,
			\item $E(u_0)>0$ and $W(u_0) + \sqrt{2N\alpha E(u_0) J(u_0)}<0$,
		\end{itemize}
		where $J$ and $W$ are as in \eqref{defi-J-W}, then there exists $a_*=a_*(\|u_0\|_{H^1})>0$ such that for all $0<a<a_*$, the corresponding solution to \eqref{DNLS} blows up in finite time.
	\end{theorem}
	
	Finally, we have the existence of finite time blow-up solutions to the focusing problem \eqref{NLS} in the energy-critical case. 
	\begin{theorem} \label{theo-blow-ener-sigma} 
		Let $N\geq 3$, $\alpha=\frac{4}{N-2}$ and $\mu=-1$. If $u_0 \in \Sigma$ satisfies one of the following conditions:
		\begin{itemize}
			\item $E(u_0)<0$,
			\item $E(u_0) =0$ and $V(u_0)<0$,
			\item $E(u_0)>0$ and $V(u_0) + \sqrt{2 E(u_0) I(u_0)}<0$,
		\end{itemize}
		where $I$ and $V$ are as in \eqref{defi-I-V}, then there exists $a_*=a_*(u_0)>0$ such that for all $0<a<a_*$, the corresponding solution to \eqref{DNLS} blows up in finite time.
	\end{theorem}
	
	\begin{theorem} \label{theo-blow-ener-rad}
		Let $N\geq 3$, $\alpha=\frac{4}{N-2}$ and $\mu=-1$. If $u_0 \in H^1$ is radially symmetric and satisfies one of the following conditions:
		\begin{itemize}
			\item $E(u_0)<0$,
			\item $E(u_0) =0$ and $W(u_0)<0$,
			\item $E(u_0)>0$ and $W(u_0) + \sqrt{\frac{8N}{N-2} E(u_0) J(u_0)}<0$,
		\end{itemize}
		where $J$ and $W$ are as in \eqref{defi-J-W}, then there exists $a_*=a_*(u_0)>0$ such that for all $0<a<a_*$, the corresponding solution to \eqref{DNLS} blows up in finite time.
	\end{theorem}

	\subsection{Idea of the proofs.}
	By using the change of variable 
	\begin{align} \label{chan-vari}
	v(t,x):= e^{at} u(t,x),
	\end{align}
	we see that 
	\[
	u \text{ solves } \eqref{DNLS} \Leftrightarrow v \text{ solves }\eqref{NLS},
	\]
	where
	\begin{equation} \label{NLS}
	\left\{ 
	\begin{array}{rcl}
	i\partial_t v + \Delta v &=& \mu e^{-a\alpha t} |v|^\alpha v, \quad (t,x) \in [0,\infty) \times \R^N, \\
	v(0,x)&=& u_0(x).
	\end{array}
	\right.
	\end{equation}
	Therefore, the global existence, scattering and finite time blow-up for \eqref{DNLS} are reduced to that of \eqref{NLS}. Since $t\mapsto e^{-a\alpha t}$ is bounded in $[0,\infty)$, the local well-posedness for \eqref{NLS} follows directly from the classical NLS (see e.g. \cite{Cazenave}). Moreover, local solutions satisfy the conservation of mass, i.e.
	\[
	M(v(t)) = \|v(t)\|^2_{L^2} = M(u_0),
	\]
	and the energy functional
	\[
	H(v(t)) :=\frac{1}{2} \|\nabla v(t)\|^2_{L^2} +\frac{\mu}{\alpha+2} e^{-a\alpha t} \|v(t)\|^{\alpha+2}_{L^{\alpha+2}}
	\]
	satisfies
	\begin{align} \label{deri-ener}
	\frac{d}{dt} H(v(t)) = -\frac{a\alpha\mu}{\alpha+2} e^{-a\alpha t} \|v(t)\|^{\alpha+2}_{L^{\alpha+2}}
	\end{align}
	for any $t$ in the existence time. 
	
	Due to the exponential decay in front of the nonlinearity, Strichartz estimates and the standard continuity argument imply the global well-posedness and scattering in the energy-critical case once the damping parameter is large. 
	
	The proofs of blow-up criteria for the focusing problem \eqref{NLS} are inspired by classical arguments of Glassey \cite{Glassey} (for data in weighted $L^2$ space) and Ogawa-Tsutsumi \cite{OT} using \eqref{deri-ener} (for radially symmetric data). The main difficulty comes from the fact that the energy functional is no longer conserved. It is in fact increasing in time. We have from \eqref{deri-ener} that
	\begin{align} \label{ener-iden}
	H(v(t)) = E(u_0) + \frac{a\alpha}{\alpha+2} \int_0^t e^{-a\alpha s} \|v(s)\|^{\alpha+2}_{L^{\alpha+2}} ds 
	\end{align}
	for all $t$ in the existence time. Thanks to the exponential decay in time, Sobolev embedding $H^1 \subset L^{\alpha+2}$ and the fact that $H^1$-norm of solution depends on initial data, we see that for a fixed time, the second term in the right hand side of \eqref{ener-iden} becomes small once $a$ tends to zero. Using this fact, we are able to show the existence of finite time blow-up solutions to \eqref{NLS}. We also refer the interested reader to \cite{KO, Ozsari} for finite time blow-up of \eqref{DNLS} in 1D.
	
	This paper is organized as follows. In Section $\ref{S2}$, we prove the global existence and scattering for the energy-critical \eqref{DNLS} with large damping parameter. In Section $\ref{S3}$, we give the proofs of blow-up criteria given in Theorems $\ref{theo-blow-mass-sigma}$--$\ref{theo-blow-ener-rad}$.

	\section{Global existence and scattering in the energy-critical case}
	\label{S2}
	\setcounter{equation}{0}
	We first recall the local well-posedness for \eqref{NLS}. Since $t\mapsto e^{-a\alpha t}$ is bounded in $[0,\infty)$, the local well-posedness for \eqref{NLS} follows directly from that of the classical NLS (see e.g. \cite{Cazenave}).
	
	\begin{lemma} [LWP in the energy-subcritical case \cite{Cazenave}] Let $N\geq 1$, $a>0$, $0<\alpha<\alpha^*$ and $\mu \in \{\pm 1\}$. Let $u_0 \in H^1$. Then there exist $T^* \in (0,\infty]$ and a unique solution to \eqref{NLS} satisfying
	\[
	v \in C([0,T^*), H^1) \cap L^q_{\loc}([0,T^*),W^{1,r})
	\]
	for any Schr\"odinger admissible pair $(q,r)$. The maximal time of existence satisfies the blow-up alternative: if $T^*<\infty$, then $\lim_{t\rightarrow T^*} \|v(t)\|_{H^1} =\infty$. Moreover, there is conservation of mass, and the energy satisfies \eqref{deri-ener} for all $t\in [0,T^*)$.
	\end{lemma}
	
	\begin{lemma} [LWP in the energy-critical case \cite{Cazenave}] Let $N\geq 3$, $a>0$, $\alpha=\frac{4}{N-2}$ and $\mu \in \{\pm 1\}$. Let $u_0 \in H^1$. Then there exist $T^* \in (0,\infty]$ and a unique solution to \eqref{NLS} satisfying
	\[
	v \in C([0,T^*),H^1) \cap L^q_{\loc}([0,T^*),W^{1,r})
	\]
	for any Schr\"odinger admissible pair $(q,r)$. The maximal time of existence satisfies if $T^*<\infty$, then $\|v\|_{L^q((0,T^*),W^{1,r})}=\infty$ for any Schr\"odinger admissible pair $(q,r)$ satisfying $2<r<N$. Moreover, there is conservation of mass, and the energy satisfies \eqref{deri-ener} for all $t\in [0,T^*)$.
	\end{lemma}

	Thanks to \eqref{chan-vari}, Theorem $\ref{theo-scat-ener}$ follows directly from the following result.
	\begin{lemma}
		Let $N\geq 3$, $\alpha=\frac{4}{N-2}$ and $\mu\in \{\pm 1\}$. Let $u_0 \in H^1$. Then there exists $a^*=a^*(u_0)>0$ such that for all $a>a^*$, the corresponding solution to \eqref{NLS} exists globally in time and scatters in $H^1$, i.e. there exists $u_+\in H^1$ such that 
			\[
			\lim_{t\rightarrow \infty} \|v(t) - e^{it\Delta} u_+\|_{H^1} =0.
			\]
	\end{lemma}
	
	\begin{proof}
		Let $T^*$ be the maximal time of existence. Fix $t_0>0$ satisfying $2t_0<T^*$. Denote		
		\begin{align} \label{defi-gam-rho}
		\gamma:= \frac{2N}{N-2}, \quad \rho:= \frac{2N^2}{N^2-2N+4}, \quad n:=\frac{2N^2}{N^2-4N+4}.
		\end{align}
		Note that $(\gamma,\rho)$ is a Schr\"odinger admissible pair and $W^{1,\rho} \subset L^n$. Using the Duhamel formula
		\[
		v(t) = e^{it\Delta} v(t_0) -\mu \int_{t_0}^t e^{i(t-s)\Delta} e^{-\frac{4a}{N-2} s} |v(s)|^{\frac{4}{N-2}} v(s) ds 
		\]
		and Strichartz estimates, we have for any $T\in (t_0,T^*)$,
		\begin{align*}
		\|v\|_{L^\gamma((t_0,T),W^{1,\rho})} &\leq \|e^{it\Delta} v(t_0)\|_{L^\gamma((t_0,T),W^{1,\rho})} + \|e^{-\frac{4a}{N-2} t} |v|^{\frac{4}{N-2}} v\|_{L^{\rho'}((t_0,T),W^{1,\rho'})} \\
		&\lesssim \|v(t_0)\|_{H^1} + \|e^{-at} v\|^{\frac{4}{N-2}}_{L^\rho((t_0,T),L^n)} \|v\|_{L^\rho((t_0,T),W^{1,\rho})} \\
		&\lesssim \|v(t_0)\|_{H^1} + e^{-\frac{4a}{N-2} t_0} \|v\|^{\frac{N+2}{N-2}}_{L^\gamma((t_0,T),W^{1,\rho})}.
		\end{align*}
		By taking $a>0$ sufficiently large depending on $t_0$, the continuity argument implies that for any $T\in (t_0,T^*)$,
		\[
		\|v\|_{L^\gamma((t_0,T),W^{1,\rho})} \leq C\|v(t_0)\|_{H^1},
		\]
		where the constant $C>0$ is independent of $T$. Letting $T\rightarrow T^*$, we obtain
		\[
		\|v\|_{L^\gamma((t_0,T^*),\dot{W}^{1,\rho})} \leq C\|v(t_0)\|_{H^1}.
		\]
		Since $v \in L^\gamma_{\loc}([0,T^*),W^{1,\rho})$, we infer that there exists $a^*=a^*(u_0)>0$ such that for all $a>a^*$,
		\[
		\|v\|_{L^\gamma([0,T^*),W^{1,\rho})} \leq C(u_0)
		\]
		which, by the blow-up alternative, implies $T^*=\infty$. Here we note that $t_0$ and $\|v(t_0)\|_{H^1}$ depend not only on $\|u_0\|_{H^1}$ but also on the profile of $u_0$. The above uniform bound also gives the scattering. In fact, let $0<t_1<t_2$. By Strichartz estimates,
		\begin{align*}
		\|e^{-it_2\Delta} v(t_2) - e^{-it_1 \Delta} v(t_1)\|_{H^1} &= \left\| \int_{t_1}^{t_2} e^{-is\Delta} e^{-\frac{4a}{N-2} s} |v(s)|^{\frac{4}{N-2}} v(s) ds \right\|_{H^1} \\
		&\lesssim \|e^{-\frac{4a}{N-2} t} |v|^{\frac{4}{N-2}} v\|_{L^{\gamma'}((t_1,t_2),W^{1,\rho'})} \\
		&\lesssim \|e^{-\frac{4}{N-2}t} v\|^{\frac{4}{N-2}}_{L^\gamma((t_1,t_2),L^n)} \|v\|_{L^\gamma((t_1,t_2),W^{1,\rho})} \\
		&\lesssim \|v\|^{\frac{N+2}{N-2}}_{L^\gamma((t_1,t_2),W^{1,\rho})} \rightarrow 0
		\end{align*}
		as $t_1,t_2 \rightarrow \infty$. This shows that $(e^{-it\Delta} v(t))_t$ is a Cauchy sequence in $H^1$. Thus, the limit
		\[
		u_+:= u_0 - i \mu \int_0^\infty e^{-is\Delta} e^{-\frac{4a}{N-2} s} |v(s)|^{\frac{4}{N-2}} v(s) ds
		\]
		exists in $\dot{H}^1$. Repeating the above arguments, we prove
		\[
		\lim_{t\rightarrow \infty} \|v(t)-e^{it\Delta} u_+\|_{H^1}=0.
		\]
		The proof is complete.
	\end{proof}

	\section{Finite time blow-up}
	\label{S3}
	\setcounter{equation}{0}
	In this section, we give the proof of blow-up criteria given in Theorems $\ref{theo-blow-mass-sigma}$--$\ref{theo-blow-ener-rad}$. To this end, we first derive some localized virial estimates related to \eqref{NLS}. 
	\subsection{Localized virial estimates}
	Given a real-valued function $\chi$, we define the virial action associated to \eqref{NLS} by
	\begin{align} \label{defi-viri-acti}
	V_\chi(t):= \int \chi(x) |v(t,x)|^2 dx.
	\end{align}
	
	\begin{lemma} [\cite{TVZ}] 
		Let $N\geq 1$, $a>0$, $0<\alpha \leq \alpha^*$ and $\mu=-1$. Let $\chi: \R^N\rightarrow \R$ be a sufficiently smooth and decaying function. Let $v$ be a $H^1$ solution to \eqref{NLS}. Then for any $t\in [0,T^*)$,
		\begin{align} \label{firs-deri-viri}
		\frac{d}{dt} V_\chi(t) = 2 \int \nabla \chi \cdot \ima \left( \nabla v(t) \overline{v}(t)\right) dx
		\end{align}
		and
		\begin{align} \label{seco-deri-viri}
		\begin{aligned}
		\frac{d^2}{dt^2}V_\chi(t) = - \int \Delta^2 \chi |v(t)|^2 dx &+ 4\sum_{j,k=1}^N \int \partial^2_{jk}\chi \rea \left( \partial_j v(t) \partial_k \overline{v}(t)\right) dx \\
		&-\frac{2\alpha}{\alpha+2} e^{-a\alpha t} \int \Delta \chi |v(t)|^{\alpha+2} dx.
		\end{aligned}
		\end{align}
	\end{lemma}
	
	A direct consequence of the above result with $\chi(x)=|x|^2$ is the following virial identity.
	\begin{corollary} \label{coro-viri-iden}
		Let $N\geq 1$, $a>0$, $0<\alpha \leq \alpha^*$ and $\mu=-1$. Let $u_0 \in \Sigma$. Then the corresponding solution to \eqref{NLS} satisfies $v \in C([0,T^*), \Sigma)$ and 
		\[
		\frac{d^2}{dt^2} \|x v(t)\|^2_{L^2} = 8 \|\nabla v(t)\|^2_{L^2} - \frac{4N\alpha}{\alpha+2} e^{-a\alpha t} \|v(t)\|^{\alpha+2}_{L^{\alpha+2}}
		\]
		for all $t\in [0,T^*)$. 
	\end{corollary}
	
	To study the blow-up criteria for \eqref{NLS} with radially symmetric initial data, we need the following localized virial estimates.
	\begin{lemma} \label{lem-viri-est-super}
		Let $N\geq 2$, $a>0$, $0<\alpha \leq 4$ and $\mu=-1$. Let $\chi_R$ be as in \eqref{defi-chi-R-super}. Let $v: [0,T^*) \times \R^N \rightarrow \C$ be a radially symmetric $H^1$ solution to \eqref{NLS}. Then for any $\vareps>0$, any $R>0$ and any $t \in [0,T^*)$,
		\begin{align} \label{loca-viri-est-super}
		\begin{aligned}
		\frac{d^2}{dt^2} V_{\chi_R}(t) \leq 8 \|\nabla v(t)\|^2_{L^2} &-\frac{4N\alpha}{\alpha+2} e^{-a\alpha t} \|v(t)\|^{\alpha+2}_{L^{\alpha+2}} \\
		&+ \left\{
		\renewcommand*{\arraystretch}{1.3}
		\begin{array}{cl}
		O \left(R^{-2} + R^{-2(N-1)}\|\nabla v(t)\|^2_{L^2}\right) &\text{if } \alpha=4, \\
		O \left(R^{-2} + \vareps^{-\frac{\alpha}{4-\alpha}} R^{-\frac{2(N-1)\alpha}{4-\alpha}} + \vareps \|\nabla v(t)\|^2_{L^2} \right) &\text{if } \alpha<4.
		\end{array}
		\right.
		\end{aligned}
		\end{align}
		Here the implicit constant depends only on $N, \alpha$ and $\|u_0\|_{L^2}$. 
	\end{lemma}
	
	Since $t\mapsto e^{-a\alpha t}$ is bounded on $[0,T^*)$, the proof of this result follows by the same lines as in \cite[Lemma 3.4]{Dinh-blow}.
	
	\begin{remark}
		The restriction $\alpha \leq 4$ comes from the Young inequality (see \cite{Dinh-blow}). If we consider $\frac{4}{N} \leq \alpha \leq \alpha^*$, then this restriction only effects the validity of $\alpha$ in $2D$.
	\end{remark}
	
	We also need the following refined version of Lemma $\ref{lem-viri-est-super}$ in the mass-critical case.
	\begin{lemma} \label{lem-viri-est-mass}
		Let $N\geq 2$, $a>0$, $\alpha=\frac{4}{N}$ and $\mu=-1$. Let $\chi_R$ be as in \eqref{defi-chi-R-super}. Let $v:[0,T^*) \times \R^N \rightarrow \C$ be a radially symmetric $H^1$ solution to \eqref{NLS}. Then for any $\vareps>0$, any $R>0$ and any $t\in [0,T^*)$,
		\begin{align} \label{loca-viri-est-mass}
		\begin{aligned}
		\frac{d^2}{dt^2} V_{\chi_R}(t) \leq 16 H(v(t)) &- 4 \int \left(\chi_{1,R} - C \vareps \chi_{2,R}^{\frac{N}{2}}\right) |\nabla v(t)|^2 dx \\
		&+ O \left( R^{-2} +\vareps R^{-2} + \vareps^{-\frac{1}{N-1}} R^{-2}\right),
		\end{aligned}
		\end{align}
		for some constant $C>0$, where
		\begin{align} \label{defi-chi-12-R}
		\chi_{1,R}:=2-\chi''_R, \quad \chi_{2,R}=2N-\Delta \chi_R.
		\end{align}
		Here the implicit constant depends only on $N, \alpha$ and $\|u_0\|_{L^2}$.
	\end{lemma}
	
	We refer the reader to \cite[Lemma 3.7]{Dinh-blow} for the proof of this result.
	
	\subsection{Finite time blow-up} ~ 
	
	{\bf $\blacktriangleright$ Mass-critical case.} In this paragraph, we give the proofs of Theorem $\ref{theo-blow-mass-sigma}$ and Theorem $\ref{theo-blow-mass-rad}$. Thanks to \eqref{chan-vari}, it suffices to prove the following blow-up criteria for the focusing problem \eqref{NLS} in the mass-critical case.
	
	\begin{lemma} \label{lem-blow-mass-sigm}
		Let $N\geq 1$, $\alpha=\frac{4}{N}$ and $\mu=-1$. If $u_0 \in \Sigma=H^1 \cap L^2(|x|^2 dx)$ satisfies one of the following conditions:
			\begin{itemize}
				\item $E(u_0)<0$, 
				\item $E(u_0) =0$ and $V(u_0)<0$, 
				\item $E(u_0)>0$ and $V(u_0) + \sqrt{2E(u_0) I(u_0)} <0$,
			\end{itemize}
			where $I$ and $V$ are as in \eqref{defi-I-V}, then there exists $a_*=a_*(\|u_0\|_{H^1})>0$ such that for all $0<a<a_*$, the corresponding solution to \eqref{NLS} blows up in finite time, i.e. $T^*<\infty$. 
	\end{lemma} 
		
	\begin{lemma} \label{lem-blow-mass-rad}
		Let $N\geq 2$, $\alpha=\frac{4}{N}$ and $\mu=-1$. If $u_0 \in H^1$ is radially symmetric and satisfies one of the following conditions:
			\begin{itemize}
				\item $E(u_0)<0$, 
				\item $E(u_0) =0$ and $W(u_0)<0$, 
				\item $E(u_0)>0$ and $W(u_0) + \sqrt{8E(u_0) J(u_0)} <0$,
			\end{itemize}
			where $J$ and $W$ are as in \eqref{defi-J-W}, then there exists $a_*=a_*(\|u_0\|_{H^1})>0$ such that for all $0<a<a_*$, the corresponding solution to \eqref{NLS} blows up in finite time. 
	\end{lemma}

	\noindent {\bf Proof of Lemma $\ref{lem-blow-mass-sigm}$.} 
	Let $u_0 \in \Sigma$ satisfy one of the conditions given in Lemma $\ref{lem-blow-mass-sigm}$. Assume by contradiction that the corresponding solution to \eqref{NLS} exists globally in time, i.e. $T^*=\infty$. By Corollary $\ref{coro-viri-iden}$ with $\alpha=\frac{4}{N}$ and \eqref{deri-ener}, we see that 
	\[
	\frac{d^2}{dt^2} \|xv(t)\|^2_{L^2} = 16 H(v(t)) = 16 E(u_0) + \frac{32a}{N+2} \int_0^t e^{-\frac{4a}{N}s} \|v(s)\|^{\frac{4}{N}+2}_{L^{\frac{4}{N}+2}} ds
	\]
	for all $t\in [0,\infty)$. It follows that
	\[
	\|xv(t)\|^2_{L^2} = I(u_0) + 4 V(u_0) t+ 8E(u_0) t^2 + \frac{32a}{N+2} \int_0^t \int_0^s \int_0^\tau e^{-\frac{4a}{N}\sigma} \|v(\sigma)\|^{\frac{4}{N}+2}_{L^{\frac{4}{N}+2}} d\sigma d\tau ds=: f(t) + A(t)
	\]
	for all $t\in [0,\infty)$,	where
	\begin{align} \label{defi-f-mass} 
	f(t):= I(u_0) + 4 V(u_0) t+ 8E(u_0) t^2
	\end{align}
	and
	\begin{align} \label{defi-A-mass}
	A(t):=\frac{32a}{N+2} \int_0^t \int_0^s \int_0^\tau e^{-\frac{4a}{N}\sigma} \|v(\sigma)\|^{\frac{4}{N}+2}_{L^{\frac{4}{N}+2}} d\sigma d\tau ds
	\end{align}
	Under the assumptions of Lemma $\ref{lem-blow-mass-sigm}$, there exists $t_0>0$ such that $f(t_0)<0$. Since $v\in C([0,\infty), H^1)$, the Sobolev embedding implies
	\[
	\sup_{t\in [0,t_0]} \|v(t)\|_{L^{\frac{4}{N}+2}} \leq \sup_{t\in [0,t_0]} \|v(t)\|_{H^1} \leq C(t_0).
	\]
	We infer that
	\[
	A(t_0) \leq C(N,t_0) a \left|\int_0^{t_0} \int_0^s \int_0^\tau e^{-\frac{4a}{N} \sigma} d\sigma d\tau ds \right| = C(N, t_0) a \left( 1- \frac{4a}{N} t_0 - \frac{16a^2}{N^2} t_0^2 - e^{-\frac{4a}{N}t_0}\right) \rightarrow 0
	\]
	as $a \rightarrow 0$. There thus exists $a_*=a_*(t_0)>0$ such that for all $0<a<a_*$,
	\[
	A(t_0) \leq -\frac{f(t_0)}{2}.
	\]
	Since $t_0$ depends only on $\|u_0\|_{H^1}$, we have proved that there exists $a_*=a_*(\|u_0\|_{H^1})>0$ such that for all $0<a<a_*$,
	\[
	\|xv(t_0)\|^2_{L^2} \leq \frac{f(t_0)}{2} <0
	\]
	which is a contradiction. The proof is complete.
	\hfill $\Box$
	
	\noindent {\bf Proof of Lemma $\ref{lem-blow-mass-rad}$.}
	Let $u_0 \in H^1$ be radially symmetric and satisfy one of the conditions given in Lemma $\ref{lem-blow-mass-rad}$. Assume by contradiction that the corresponding solution to \eqref{NLS} exists globally in time, i.e. $T^*=\infty$. It is well-known that the corresponding solution to \eqref{NLS} is radially symmetric. By Lemma $\ref{lem-viri-est-mass}$ and \eqref{deri-ener}, we see that for any $\vareps>0$, any $R>0$ and any $t \in [0,\infty)$,
	\begin{align*}
	\frac{d^2}{dt^2} V_{\chi_R}(t) &\leq 16 H(v(t)) - 4 \int \left(\chi_{1,R} - C \vareps \chi_{2,R}^{\frac{N}{2}}\right) |\nabla v(t)|^2 dx  + O \left( R^{-2} +\vareps R^{-2} + \vareps^{-\frac{1}{N-1}} R^{-2}\right) \\
	&= 16 E(u_0) + \frac{32a}{N+2} \int_0^t e^{-\frac{4a}{N}s} \|v(s)\|^{\frac{4}{N}+2}_{L^{\frac{4}{N}+2}} ds  - 4 \int \left(\chi_{1,R} - C \vareps \chi_{2,R}^{\frac{N}{2}}\right) |\nabla v(t)|^2 dx \\
	&\mathrel{\phantom{= 16 E(u_0) + \frac{32a}{N+2} \int_0^t e^{-\frac{4a}{N}s} \|v(s)\|^{\frac{4}{N}+2}_{L^{\frac{4}{N}+2}} ds }} + O \left( R^{-2} +\vareps R^{-2} + \vareps^{-\frac{1}{N-1}} R^{-2}\right).
	\end{align*}
	Assume for the moment that 
	\begin{align} \label{posi-chi-12-R}
	\chi_{1,R} - C\vareps \chi_{2,R}^{\frac{N}{2}} \geq 0, \quad \forall r\geq 0
	\end{align}
	for a sufficiently small $\vareps>0$. We will consider separtely three cases: $E(u_0)<0$, $E(u_0)=0$ and $E(u_0)>0$.
	
	$\bullet$ If $E(u_0)<0$, then by choosing $R>0$ large enough depending on $\vareps$, we see that 
	\[
	\frac{d^2}{dt^2} V_{\chi_R}(t) \leq 12 E(u_0) + \frac{32a}{N+2} \int_0^t e^{-\frac{4a}{N}s} \|v(s)\|^{\frac{4}{N}+2}_{L^{\frac{4}{N}+2}} ds
	\]
	for all $t\in [0,\infty)$.  It follows that
	\[
	V_{\chi_R}(t) \leq J(u_0) + 2W(u_0) t + 6E(u_0) t^2 + \frac{32a}{N+2} \int_0^t \int_0^s \int_0^\tau e^{-\frac{4a}{N} \sigma} \|v(\sigma)\|^{\frac{4}{N}+2}_{L^{\frac{4}{N}+2}} d\sigma d\tau ds =:f_1(t) + A(t)
	\]
	for all $t\in [0,\infty)$, where 
	\[
	f_1(t):=J(u_0) + 2W(u_0) t + 6E(u_0) t^2 
	\]
	and $A(t)$ is as in \eqref{defi-A-mass}. Since $E(u_0)<0$, there exists $t_1>0$ such that $f_1(t_1)<0$. As in the proof of Lemma $\ref{lem-blow-mass-sigm}$, there exists $a_*=a_*(t_1)>0$ such that for all $0<a<a_*$,
	\[
	A(t_1) \leq -\frac{f_1(t_1)}{2}.
	\]
	Since $t_1$ depends only on $\|u_0\|_{H^1}$, we prove that there exists $a_*=a_*(\|u_0\|_{H^1})>0$ such that for all $0<a<a_*$,
	\[
	V_{\chi_R}(t_1) \leq \frac{f_1(t_1)}{2}<0
	\]
	which is a contradiction. 
	
	$\bullet$ If $E(u_0)=0$, then choosing $R>0$ large enough depending on $\vareps$, we see that 
	\[
	\frac{d^2}{dt^2} V_{\chi_R}(t) \leq 2\delta + \frac{32a}{N+2} \int_0^t e^{-\frac{4a}{N}s} \|v(s)\|^{\frac{4}{N}+2}_{L^{\frac{4}{N}+2}} ds
	\]
	for all $t\in [0,\infty)$, where $\delta>0$ will be chosen shortly. It follows that
	\[
	V_{\chi_R}(t) \leq J(u_0) + 2 W(u_0) t + \delta t^2 + \frac{32a}{N+2} \int_0^t \int_0^s \int_0^\tau e^{-\frac{4a}{N} \sigma} \|v(\sigma)\|^{\frac{4}{N}+2}_{L^{\frac{4}{N}+2}} d\sigma d\tau ds =:f_2(t) + A(t)
	\]
	for all $t\in [0,\infty)$, where
	\[
	f_2(t):= J(u_0) + 2 W(u_0) t + \delta t^2 
	\]
	and $A(t)$ is as in \eqref{defi-A-mass}. In order to $f_2$ takes negative values on $[0,\infty)$, we need 
	\[
	W(u_0)<0, \quad [W(u_0)]^2- \delta J(u_0)>0.
	\] 
	By the assumption $W(u_0)<0$, we can choose $\delta>0$ small so that $[W(u_0)]^2 - \delta J(u_0)>0$. This shows that under the assuptions $E(u_0)=0$ and $W(u_0)<0$, there exists $t_2>0$ such that $f_2(t_2)<0$. Moreover, as in the proof of Lemma $\ref{lem-blow-mass-sigm}$, there exists $a_*=a_*(t_2)>0$ such that for all $0<a<a_*$,
	\[
	A(t_2) \leq -\frac{f_2(t_2)}{2}.
	\]
	Since $t_2$ depends only on $\|u_0\|_{H^1}$, we prove that there exists $a_*=a_*(\|u_0\|_{H^1})>0$ such that for all $0<a<a_*$,
	\[
	V_{\chi_R}(t_2) \leq \frac{f_2(t_2)}{2}<0
	\]
	which is a contradiction. 
	
	$\bullet$ If $E(u_0)>0$, then by choosing $R>0$ large enough depending on $\vareps$, we get
	\[
	\frac{d^2}{dt^2}V_{\chi_R}(t) \leq 16(1+\delta) E(u_0) + \frac{32a}{N+2} \int_0^t e^{-\frac{4a}{N}s} \|v(s)\|^{\frac{4}{N}+2}_{L^{\frac{4}{N}+2}} ds
	\]
	for all $t\in [0,\infty)$, where $\delta>0$ will be chosen later. We infer that
	\[
	V_{\chi_R}(t) \leq J(u_0) + 2 W(u_0) t + 8(1+\delta) E(u_0)t^2 + \frac{32a}{N+2} \int_0^t \int_0^s \int_0^\tau e^{-\frac{4a}{N} \sigma} \|v(\sigma)\|^{\frac{4}{N}+2}_{L^{\frac{4}{N}+2}} d\sigma d\tau ds=: f_3(t) + A(t)
	\]
	for all $t\in [0,\infty)$, where
	\[
	f_3(t):= J(u_0) + 2 W(u_0) t + 8(1+\delta) E(u_0)t^2
	\]
	and $A(t)$ is as in \eqref{defi-A-mass}. To ensure $f_3$ takes negative values, we need 
	\[
	W(u_0)<0, \quad [W(u_0)]^2 -8(1+\delta) E(u_0)J(u_0)>0. 
	\]
	By taking $\delta>0$ sufficiently small, the above conditions are equivalent to 
	\[
	W(u_0)<0, \quad [W(u_0)]^2 - 8 E(u_0) J(u_0)>0
	\]
	hence $W(u_0) + \sqrt{8 E(u_0) J(u_0)} <0$. Therefore, under the assumptions of Lemma $\ref{lem-blow-mass-rad}$, there exists $t_3>0$ such that $f_3(t_3)<0$. On the other hand, as in the proof of Lemma $\ref{lem-blow-mass-sigm}$, there exists $a_*=a_*(t_3)>0$ such that for all $0<a<a_*$,
	\[
	A(t_3) \leq -\frac{f_3(t_3)}{2}.
	\]
	Since $t_3$ depends only on $\|u_0\|_{H^1}$, we prove that there exists $a_*=a_*(\|u_0\|_{H^1})>0$ such that for all $0<a<a_*$,
	\[
	V_{\chi_R}(t_3) \leq \frac{f_3(t_3)}{2}<0
	\]
	which is a contradiction. 
	
	It remains to show that \eqref{posi-chi-12-R} holds under the choice of $\chi_R$ as in \eqref{defi-vartheta}--\eqref{defi-chi-R}. We see that
	\[
	\chi'_R(r) = R \theta'(r/R) = R \vartheta (r/R), \quad \chi''_R(r) = \vartheta'(r/R).
	\]
	Using the fact $\Delta \chi_R(x) = \chi''_R(r) + \frac{N-1}{r}\chi'_R(r)$, we have
	\[
	\chi_{1,R}=2-\chi''_R, \quad \chi_{2,R} = 2- \chi''_R + (N-1) \left(2-\frac{\chi'_R}{r}\right).
	\]
	
	For $0\leq r \leq R$, \eqref{posi-chi-12-R} holds trivially since $\chi_{1,R} = \chi_{2,R} =0$ on $0\leq r \leq R$.
	
	For $R<r \leq (1+1/\sqrt{3}) R$, we have $\chi_{1,R} = 6(r/R -1)^2$ and 
	\begin{align*}
	\chi_{2,R} = 6(r/R-1)^2 + 2(N-1) \frac{(r/R-1)^3}{r/R} &= 6(r/R-1)^2 \left(1+\frac{(N-1)(r/R-1)}{3r/R}\right) \\
	&<6(r/R-1)^2 \left(1+\frac{N-1}{3\sqrt{3}}\right).
	\end{align*}
	By choosing $\vareps>0$ small enough, we see that \eqref{posi-chi-12-R} holds.
	
	For $r>(1+1/\sqrt{3})R$, the fact $\vartheta'(r/R)\leq 0$ implies that $\chi_{1,R} = 2- \chi''_R \geq 2$. On the other hand, $\chi_{2,R} \leq C$ for some constant $C>0$. It follows that \eqref{posi-chi-12-R} holds by choosing $\vareps>0$ small enough.
	
	Collecting the above cases, we prove \eqref{posi-chi-12-R}. The proof is complete.
	\hfill $\Box$

	{\bf $\blacktriangleright$ Mass-supercritical and energy-subcritical case.} In this paragraph, we give the proofs of Theorem $\ref{theo-blow-super-sigma}$ and Theorem $\ref{theo-blow-super-rad}$. By the change of variable \eqref{chan-vari}, the proofs are reduced to prove the following results.
	
	\begin{lemma} \label{lem-blow-super-sigma}
		Let $N\geq 1$, $\frac{4}{N}<\alpha<\alpha^*$ and $\mu=-1$. If $u_0 \in \Sigma$ satisfies one of the following conditions:
		\begin{itemize}
			\item $E(u_0)<0$,
			\item $E(u_0) =0$ and $V(u_0)<0$,
			\item $E(u_0)>0$ and $V(u_0) + \sqrt{2 E(u_0) I(u_0)}<0$,
		\end{itemize}
		where $I$ and $V$ are as in \eqref{defi-I-V}, then there exists $a_*=a_*(\|u_0\|_{H^1})>0$ such that for all $0<a<a_*$, the corresponding solution to \eqref{NLS} blows up in finite time.
	\end{lemma}

	\begin{lemma} \label{lem-blow-super-rad}
		Let $\mu=-1$ and
		\[
		\left\{
		\begin{array}{ccc}
		\frac{4}{N}<\alpha<\frac{4}{N-2} &\text{if}& N\geq 3, \\
		2<\alpha \leq 4 &\text{if} & N=2.
		\end{array}
		\right.
		\] 
		If $u_0 \in H^1$ is radially symmetric and satisfies one of the following conditions:
		\begin{itemize}
			\item $E(u_0)<0$,
			\item $E(u_0) =0$ and $W(u_0)<0$,
			\item $E(u_0)>0$ and $W(u_0) + \sqrt{2N\alpha E(u_0) J(u_0)}<0$,
		\end{itemize}
		where $J$ and $W$ are as in \eqref{defi-J-W}, then there exists $a_*=a_*(\|u_0\|_{H^1})>0$ such that for all $0<a<a_*$, the corresponding solution to \eqref{NLS} blows up in finite time.
	\end{lemma}

	\noindent {\bf Proof of Lemma $\ref{lem-blow-super-sigma}$.}
	Assume by contradiction that the corresponding solution to \eqref{NLS} exists globally in time. By Corollary $\ref{coro-viri-iden}$, \eqref{deri-ener} and the fact $N\alpha>4$, we have
	\begin{align*}
	\frac{d^2}{dt^2} \|x v(t)\|^2_{L^2} &= 8 \|\nabla v(t)\|^2_{L^2} - \frac{4N\alpha}{\alpha+2} e^{-a\alpha t} \|v(t)\|^{\alpha+2}_{L^{\alpha+2}} \\
	&= 16H(u(t)) - \frac{4(N\alpha-4)}{\alpha+2} e^{-a\alpha t} \|v(t)\|^{\alpha+2}_{L^{\alpha+2}} \\
	&\leq 16 H(v(t)) \\
	&= 16 E(u_0) + \frac{16a\alpha}{\alpha+2} \int_0^t e^{-a\alpha s} \|v(s)\|^{\alpha+2}_{L^{\alpha+2}} ds
	\end{align*}
	for all $t\in [0,\infty)$. It follows that
	\[
	\|xv(t)\|^2_{L^2} \leq I(u_0) + 4 V(u_0) t + 8 E(u_0) t^2 + \frac{16a \alpha}{\alpha+2} \int_0^t \int_0^s \int_0^\tau e^{-a \alpha \sigma} \|v(\sigma)\|^{\alpha+2}_{L^{\alpha+2}} d\sigma d\tau ds=: f(t) + B(t)
	\]
	for all $t\in [0,\infty)$, where $f(t)$ is as in \eqref{defi-f-mass} and 
	\[
	B(t):= \frac{16a \alpha}{\alpha+2} \int_0^t \int_0^s \int_0^\tau e^{-a \alpha \sigma} \|v(\sigma)\|^{\alpha+2}_{L^{\alpha+2}} d\sigma d\tau ds
	\]
	Under the assumptions of Lemma $\ref{lem-blow-super-sigma}$, there exists $t_0>0$ such that $f(t_0)<0$. Since $v\in C([0,\infty), H^1)$, the Sobolev embedding implies
	\[
	\sup_{t\in [0,t_0]} \|v(t)\|_{L^{\alpha+2}} \leq \sup_{t\in [0,t_0]} \|v(t)\|_{H^1} \leq C(t_0).
	\]
	We infer that
	\[
	B(t_0) \leq C(\alpha, t_0) a \left|\int_0^{t_0} \int_0^s \int_0^\tau e^{-a \alpha \sigma} d\sigma d\tau ds \right| = C(\alpha, t_0)a \left( 1- a\alpha t_0 - a^2 \alpha^2 t_0^2 - e^{-a\alpha t_0}\right) \rightarrow 0
	\]
	as $a \rightarrow 0$. There thus exists $a_*=a_*(t_0)>0$ such that for all $0<a<a_*$,
	\[
	B(t_0) \leq -\frac{f(t_0)}{2}.
	\]
	Since $t_0$ depends only on $\|u_0\|_{H^1}$, we have proved that there exists $a_*=a_*(\|u_0\|_{H^1})>0$ such that for all $0<a<a_*$,
	\[
	\|xv(t_0)\|^2_{L^2} \leq \frac{f(t_0)}{2} <0
	\]
	which is a contradiction. The proof is complete.
	\hfill $\Box$
	
	\noindent {\bf Proof of Lemma $\ref{lem-blow-super-rad}$.}
	Let $u_0 \in H^1$ be radially symmetric and satisfy one of the conditions stated in Lemma $\ref{lem-blow-super-rad}$. Assume by contradiction that the corresponding solution to \eqref{NLS} exists globally in time. By Lemma $\ref{lem-viri-est-super}$, we have for any $\vareps>0$, any $R>0$ and any $t\in [0,\infty)$,
	\begin{align*}
	\frac{d^2}{dt^2} V_{\chi_R}(t) &\leq 8 \|\nabla v(t)\|^2_{L^2} -\frac{4N\alpha}{\alpha+2} e^{-a\alpha t} \|v(t)\|^{\alpha+2}_{L^{\alpha+2}} \\
	&\mathrel{\phantom{\leq 8 \|\nabla v(t)\|^2_{L^2}}} + \left\{
	\renewcommand*{\arraystretch}{1.3}
	\begin{array}{cl}
	O \left(R^{-2} + R^{-2(N-1)}\|\nabla v(t)\|^2_{L^2}\right) &\text{if } \alpha=4 \\
	O \left(R^{-2} + \vareps^{-\frac{\alpha}{4-\alpha}} R^{-\frac{2(N-1)\alpha}{4-\alpha}} + \vareps \|\nabla v(t)\|^2_{L^2} \right) &\text{if } \alpha<4
	\end{array}
	\right. \\
	&= 4N\alpha H(v(t)) - 2(N\alpha-4) \|\nabla v(t)\|^2_{L^2} \\
	&\mathrel{\phantom{\leq 8 \|\nabla v(t)\|^2_{L^2}}}  + \left\{
	\renewcommand*{\arraystretch}{1.3}
	\begin{array}{cl}
	O \left(R^{-2} + R^{-2(N-1)}\|\nabla v(t)\|^2_{L^2}\right) &\text{if } \alpha=4,\\
	O \left(R^{-2} + \vareps^{-\frac{\alpha}{4-\alpha}} R^{-\frac{2(N-1)\alpha}{4-\alpha}} + \vareps \|\nabla v(t)\|^2_{L^2} \right) &\text{if } \alpha<4.
	\end{array}
	\right.
	\end{align*}
	By \eqref{deri-ener}, we have for any $\vareps>0$, any $R>0$ and any $t\in [0,\infty)$,
	\begin{align*}
	\frac{d^2}{dt^2} V_{\chi_R}(t) \leq 4N\alpha E(u_0) & + \frac{4Na\alpha^2}{\alpha+2} \int_0^t e^{-a\alpha s} \|v(s)\|^{\alpha+2}_{L^{\alpha+2}}ds - 2(N\alpha-4) \|\nabla v(t)\|^2_{L^2} \\
	&+ \left\{
	\renewcommand*{\arraystretch}{1.3}
	\begin{array}{cl}
	O \left(R^{-2} + R^{-2(N-1)}\|\nabla v(t)\|^2_{L^2}\right) &\text{if } \alpha=4, \\
	O \left(R^{-2} + \vareps^{-\frac{\alpha}{4-\alpha}} R^{-\frac{2(N-1)\alpha}{4-\alpha}} + \vareps \|\nabla v(t)\|^2_{L^2} \right) &\text{if } \alpha<4.
	\end{array}
	\right.
	\end{align*}
	We consider separately three cases: $E(u_0)<0$, $E(u_0)=0$ and $E(u_0)>0$.
	
	$\bullet$ If $E(u_0)<0$, then by choosing $R>0$ large enough in the case $\alpha=4$, and $\vareps>0$ small enough and then $R>0$ large enough depending on $\vareps$ in the case $\alpha<4$, we see that 
	\[
	\frac{d^2}{dt^2} V_{\chi_R} (t) \leq 2N\alpha E(u_0) + \frac{4Na\alpha^2}{\alpha+2} \int_0^t e^{-a\alpha s} \|v(s)\|^{\alpha+2}_{L^{\alpha+2}}ds
	\]
	for all $t\in [0,\infty)$. We infer that
	\[
	V_{\chi_R}(t) \leq J(u_0) + 2W(u_0) t + N\alpha E(u_0) t^2 + \frac{4Na\alpha^2}{\alpha+2} \int_0^t \int_0^s \int_0^\tau e^{-a\alpha \sigma} \|v(\sigma)\|^{\alpha+2}_{L^{\alpha+2}}d\sigma d\tau ds =:f_1(t) + C(t)
	\]
	for all $t\in [0,\infty)$, where
	\[
	f_1(t):= J(u_0) + 2W(u_0) t + N\alpha E(u_0) t^2
	\]
	and
	\begin{align} \label{defi-C}
	C(t):=\frac{4Na\alpha^2}{\alpha+2} \int_0^t \int_0^s \int_0^\tau e^{-a\alpha \sigma} \|v(\sigma)\|^{\alpha+2}_{L^{\alpha+2}}d\sigma d\tau ds.
	\end{align}
	Since $E(u_0)<0$, there exists $t_1>0$ such that $f_1(t_1)<0$. Moreover, by the same argument as in the proof of Lemma $\ref{lem-blow-super-sigma}$, there exists $a_*=a_*(t_1)>0$ such that for all $0<a<a_*$,
	\[
	C(t_1) \leq -\frac{f_1(t_1)}{2}.
	\]
	Since $t_1$ depends only on $\|u_0\|_{H^1}$, we prove that there exists $a_*=a_*(\|u_0\|_{H^1})>0$ such that for all $0<a<a_*$,
	\[
	V_{\chi_R}(t_1) \leq \frac{f_1(t_1)}{2}<0
	\]
	which is a contradiction. 
	
	$\bullet$ If $E(u_0)=0$, then by choosing $R>0$ large enough in the case $\alpha=4$, and $\vareps>0$ small enough and then $R>0$ large enough depending on $\vareps$ in the case $\alpha<4$, we have
	\[
	\frac{d^2}{dt^2} V_{\chi_R}(t) \leq 2\delta  + \frac{4Na\alpha^2}{\alpha+2} \int_0^t e^{-a\alpha s} \|v(s)\|^{\alpha+2}_{L^{\alpha+2}}ds
	\]
	for all $t\in [0,\infty)$, where $\delta>0$ will be chosen later. It follows that
	\[
	V_{\chi_R}(t) \leq J(u_0) +2W(u_0) t + \delta t^2 + \frac{4Na\alpha^2}{\alpha+2} \int_0^t \int_0^s \int_0^\tau e^{-a\alpha \sigma} \|v(\sigma)\|^{\alpha+2}_{L^{\alpha+2}}d\sigma d\tau ds =:f_2(t) + C(t)
	\]
	for all $t\in [0,\infty)$, where 
	\[
	f_2(t):= J(u_0) +2W(u_0) t + \delta t^2
	\]
	and $C(t)$ is as in \eqref{defi-C}. Since $W(u_0)<0$, there exists $\delta>0$ small enough so that $[W(u_0)]^2 - \delta J(u_0)>0$. This shows that there exists $t_2>0$ such that $f_2(t_2)<0$. Moreover, as in the proof of Lemma $\ref{lem-blow-super-sigma}$, there exists $a_*=a_*(t_2)>0$ such that for all $0<a<a_*$,
	\[
	C(t_2) \leq -\frac{f_2(t_2)}{2}.
	\]
	Since $t_2$ depends only on $\|u_0\|_{H^1}$, we prove that there exists $a_*=a_*(\|u_0\|_{H^1})>0$ such that for all $0<a<a_*$,
	\[
	V_{\chi_R}(t_2) \leq \frac{f_2(t_2)}{2}<0
	\]
	which is a contradiction. 
	
	$\bullet$ If $E(u_0)>0$, then by choosing $R>0$ large enough in the case $\alpha=4$, and $\vareps>0$ small enough and then $R>0$ large enough depending on $\vareps$ in the case $\alpha<4$, we get that 
	\[
	\frac{d^2}{dt^2} V_{\chi_R}(t) \leq 4N\alpha(1+\delta) E(u_0) + \frac{4Na\alpha^2}{\alpha+2} \int_0^t e^{-a\alpha s} \|v(s)\|^{\alpha+2}_{L^{\alpha+2}}ds
	\]
	for all $t\in [0,\infty)$, where $\delta>0$ to be chosen shortly. It folows that
	\[
	V_{\chi_R}(t) \leq J(u_0) +2W(u_0) t + 2N\alpha(1+\delta) E(u_0) t^2 + \frac{4Na\alpha^2}{\alpha+2} \int_0^t \int_0^s \int_0^\tau e^{-a\alpha \sigma} \|v(\sigma)\|^{\alpha+2}_{L^{\alpha+2}}d\sigma d\tau ds  =:f_3(t) + C(t)
	\]
	for all $t\in [0,\infty)$, where
	\[
	f_3(t):= J(u_0) +2W(u_0) t + 2N\alpha(1+\delta) E(u_0) t^2
	\] 
	and $C(t)$ is as in \eqref{defi-C}. By the assumption $W(u_0) + \sqrt{2N\alpha E(u_0) J(u_0)} <0$, we see that there exists $\delta>0$ small enough so that
	\[
	W(u_0)<0, \quad [W(u_0)]^2 - 2N\alpha (1+\delta) E(u_0) J(u_0)>0.
	\]
	This shows the existence of $t_3>0$ such that $f_3(t_3)<0$. On the other hand, by the same reasoning as in the proof of Lemma $\ref{lem-blow-super-sigma}$, there exists $a_*=a_*(t_3)>0$ such that for all $0<a<a_*$,
	\[
	C(t_3) \leq -\frac{f_3(t_3)}{2}.
	\]
	Since $t_3$ depends only on $\|u_0\|_{H^1}$, we prove that there exists $a_*=a_*(\|u_0\|_{H^1})>0$ such that for all $0<a<a_*$,
	\[
	V_{\chi_R}(t_3) \leq \frac{f_3(t_3)}{2}<0
	\]
	which is a contradiction. 
	
	Collecting the above case, we finish the proof.	
	\hfill $\Box$
	
	{\bf $\blacktriangleright$ Energy-critical case.}
	In this paragraph, we give the proofs of Theorem $\ref{theo-blow-ener-sigma}$ and Theorem $\ref{theo-blow-ener-rad}$. Thanks to \eqref{chan-vari}, it is enough to prove the following results.
	
	\begin{lemma} \label{lem-blow-ener-sigma}
		Let $N\geq 3$, $\alpha=\frac{4}{N-2}$ and $\mu=-1$. If $u_0 \in \Sigma$ satisfies one of the following conditions:
		\begin{itemize}
			\item $E(u_0)<0$,
			\item $E(u_0) =0$ and $V(u_0)<0$,
			\item $E(u_0)>0$ and $V(u_0) + \sqrt{2 E(u_0) I(u_0)}<0$,
		\end{itemize}
		where $I$ and $V$ are as in \eqref{defi-I-V}, then there exists $a_*=a_*(u_0)>0$ such that for all $0<a<a_*$, the corresponding solution to \eqref{NLS} blows up in finite time.
	\end{lemma}

	\begin{lemma} \label{lem-blow-ener-rad}
		Let $N\geq 3$, $\alpha=\frac{4}{N-2}$ and $\mu=-1$. If $u_0 \in H^1$ is radially symmetric and satisfies one of the following conditions:
		\begin{itemize}
			\item $E(u_0)<0$,
			\item $E(u_0) =0$ and $W(u_0)<0$,
			\item $E(u_0)>0$ and $W(u_0) + \sqrt{\frac{8N}{N-2} E(u_0) J(u_0)}<0$,
		\end{itemize}
		where $J$ and $W$ are as in \eqref{defi-J-W}, then there exists $a_*=a_*(u_0)>0$ such that for all $0<a<a_*$, the corresponding solution to \eqref{NLS} blows up in finite time.
	\end{lemma}
	
	\noindent {\bf Proof of Lemma $\ref{lem-blow-ener-sigma}$.}
	The proof is similar to the one of Lemma $\ref{lem-blow-mass-sigm}$. Assume by contradiction that the solution to \eqref{NLS} exists globally in time. By Corollary $\ref{coro-viri-iden}$ and \eqref{deri-ener}, we have
	\begin{align*}
	\frac{d^2}{dt^2} \|xv(t)\|^2_{L^2} &= 8 \|\nabla v(t)\|^2_{L^2} - 8 e^{-\frac{4a}{N-2} t} \|v(t)\|^{\frac{2N}{N-2}}_{L^{\frac{2N}{N-2}}} \\
	&= 16 H(v(t)) - \frac{16}{N} e^{-\frac{4a}{N-2} t} \|v(t)\|^{\frac{2N}{N-2}}_{L^{\frac{2N}{N-2}}} \\
	&\leq 16 H(v(t)) \\
	& = 16 E(u_0) + \frac{32a}{N} \int_0^t e^{-\frac{4a}{N-2}s} \|v(s)\|^{\frac{4}{N-2}+2}_{L^{\frac{4}{N-2}+2}} ds
	\end{align*}
	for all $t\in [0,\infty)$. The proof follows by the same argument as in the one of Lemma $\ref{lem-blow-mass-sigm}$. Note that the time of existence depends not only on $\|u_0\|_{H^1}$ but also on the profile of $u_0$. We omit the details.
	\hfill $\Box$
	
	\noindent {\bf Proof of Lemma $\ref{lem-blow-ener-rad}$.}
	The proof is similar to the one of Lemma $\ref{lem-blow-super-rad}$ by using 
	\begin{align*}
	\frac{d^2}{dt^2} V_{\chi_R}(t) &\leq 8 \|\nabla v(t)\|^2_{L^2} -8e^{-\frac{4a}{N-2} t} \|v(t)\|^{\frac{2N}{N-2}}_{L^{\frac{2N}{N-2}}} \\
	&\mathrel{\phantom{\leq 8 \|\nabla v(t)\|^2_{L^2}}} + \left\{
	\renewcommand*{\arraystretch}{1.3}
	\begin{array}{cl}
	O \left(R^{-2} + R^{-4}\|\nabla v(t)\|^2_{L^2}\right) &\text{if } N=3 \\
	O \left(R^{-2} + \vareps^{-\frac{1}{N-3}} R^{-\frac{2(N-1)}{N-3}} + \vareps \|\nabla v(t)\|^2_{L^2} \right) &\text{if } N\geq 4
	\end{array}
	\right.  \\
	&=\frac{16N}{N-2} H(v(t)) - \frac{16}{N-2} \|\nabla v(t)\|^2_{L^2} \\
	&\mathrel{\phantom{=\frac{16N}{N-2} H(v(t))}} + \left\{
	\renewcommand*{\arraystretch}{1.3}
	\begin{array}{cl}
	O \left(R^{-2} + R^{-4}\|\nabla v(t)\|^2_{L^2}\right) &\text{if } N=3 \\
	O \left(R^{-2} + \vareps^{-\frac{1}{N-3}} R^{-\frac{2(N-1)}{N-3}} + \vareps \|\nabla v(t)\|^2_{L^2} \right) &\text{if } N\geq 4
	\end{array}
	\right. \\
	&= \frac{16N}{N-2} E(u_0) + \frac{32a}{N-2} \int_0^t e^{-\frac{4a}{N-2}s} \|v(s)\|^{\frac{4}{N-2}+2}_{L^{\frac{4}{N-2}+2}} ds  - \frac{16}{N-2} \|\nabla v(t)\|^2_{L^2} \\
	&\mathrel{\phantom{= \frac{16N}{N-2} E(u_0)}} + \left\{
	\renewcommand*{\arraystretch}{1.3}
	\begin{array}{cl}
	O \left(R^{-2} + R^{-4}\|\nabla v(t)\|^2_{L^2}\right) &\text{if } N=3, \\
	O \left(R^{-2} + \vareps^{-\frac{1}{N-3}} R^{-\frac{2(N-1)}{N-3}} + \vareps \|\nabla v(t)\|^2_{L^2} \right) &\text{if } N\geq 4.
	\end{array}
	\right.
	\end{align*}
	We thus omit the details.
	\hfill $\Box$

	\section*{Acknowledgement}
	This work was supported in part by the Labex CEMPI (ANR-11-LABX-0007-01). The author would like to express his deep gratitude to his wife - Uyen Cong for her encouragement and support. He also would like to thank the reviewer for his/her helpful comments and suggestions. 
	

\end{document}